\documentclass[11pt, letterpaper]{amsart}

\usepackage{amsmath,amssymb,amsthm,graphicx,mathrsfs,url,color,bbold}
\usepackage[open, openlevel=2, depth=3, atend]{bookmark}
\hypersetup{pdfstartview=XYZ}
\usepackage{cancel}
\usepackage[utf8]{inputenc}
\usepackage[T1]{fontenc}
\usepackage[english]{babel}
\DeclareMathOperator{\Tr}{Tr}
\DeclareMathOperator{\vol}{vol}

\newtheorem{theorem}{Theorem}

\newtheorem*{theorem*}{Theorem}
\newtheorem*{theoreme*}{Théorème}
\newtheorem{proposition}{Proposition}[section]

\newtheorem{definition-f}[proposition]{Définition}

\newtheorem{lemma}[proposition]{Lemma}

\newtheorem{remark}{Remark}

\newtheorem*{remark*}{Remark}

\title[Lower Bound for $\Theta$]{A lower bound for the $\Theta$ function on manifolds without conjugate points.}
\author{Yannick Bonthonneau}
\email{yannick.bonthonneau@cirget.ca}
\address{CRM, UQ\`AM, Avenue du Pr\'esident Kennedy, Montr\'eal, Qu\'ebec, Canada.}
\begin{document}

\begin{abstract}
In this short note, we prove that the usual $\Theta$ function on a Riemannian manifold without conjugate points is uniformly bounded from below. This extends a result of Green in two dimensions. 
This elementary lemma implies that the Bérard remainder in the Weyl law is valid for a manifold without conjugate points, without any restriction on the dimension.
\end{abstract}

\maketitle

\section{Introduction}

Let $(M,g)$ be a compact manifold of dimension $n$. Then the spectrum of its Laplacian is discrete. We denote the eigenvalues by $\mu_0 = 0 < \mu_1 \leq \dots$, and its counting function by
\begin{equation*}
N(\lambda) := \# \{ \mu_i \leq \lambda^2 \}.
\end{equation*}
In 1977, Bérard proved the following: 
\begin{theorem}[\cite{Berard-77}]\label{thm:old}
Assume that $n=2$ and $M$ has no conjugate points, or that $M$ has non-positive sectional curvature. Then, as $\lambda$ tends to infinity,
\begin{equation}\label{eq:Berard-remainder}
N(\lambda) = \frac{\vol(B^\ast M)}{(2\pi)^n} \lambda^n + \mathcal{O}\left( \frac{\lambda^{n-1}}{\log \lambda} \right).
\end{equation}
\end{theorem}
In this note, we prove
\begin{theorem}\label{thm:new}
It suffices to assume that $M$ does not have conjugate points to obtain the above result.
\end{theorem}

This really \emph{is} an improvement of theorem \ref{thm:old}, as there exist manifolds without conjugate points whose curvature has no sign. One can find such examples in Gulliver \cite{Gulliver-75}, or Ballmann-Brin-Burns \cite{Ballmann-Brin-Burns-87}.

We will not enter into all the details of the original proof, as we will just make an observation on a crucial point in the arguments of Bérard. To obtain the theorem, Bérard studied the local behaviour of the wave trace via the Hadamard parametrix. The kernel $K(t,x,x')$ of the wave operator $\cos t \sqrt{-\Delta}$ on $\widetilde{M}$ -- the universal cover of $M$ -- has an expansion of the form
\begin{equation}\label{eq:Hadamard}
K(t,x,x') = C_0 \sum_{k\geq 0} u_k(x,x') |t| \frac{(t^2 - d(x,x')^2)^{k- (n+1)/2}}{\Gamma(k-(n-1)/2)} \mod(C^\infty)
\end{equation}
The coefficients $u_k$ satisfy certain transport equations along the geodesic between $x$ and $x'$. The expansion \eqref{eq:Hadamard} is valid on the universal cover of $M$ as soon as $M$ has no conjugate points. A critical part of the proof of Bérard, which is the \emph{only} spot where the negative curvature assumption is used, is the lemma:
\begin{lemma}[\cite{Berard-77}]\label{lemma:old}
Let $(N,g)$ be the universal cover of a compact manifold without conjugate points. When $n=2$ or if the curvature is non-positive, for all $k\geq 0$ and $l \geq 0$,
\begin{equation*}
\Delta^l_{x'} u_k(x, x') = \mathcal{O}(1) e^{\mathcal{O}(d(x,x'))}.
\end{equation*}
\end{lemma}
To prove theorem \ref{thm:new}, it suffices to establish
\begin{lemma}\label{lemma:new}
The conclusion of lemma \ref{lemma:old} holds with the sole assumption that $(N,g)$ is a complete, simply connected Riemannian manifold without conjugate points and bounded geometry.
\end{lemma}
The assumption of bounded geometry means here that the curvature tensor and all its covariant derivatives are bounded on $N$, since $\exp_x$ is a global diffeomorphism for any $x\in N$.

\begin{remark}
In \cite{Hassell-Tacy-15}, Hassell and Tacy gave a uniform logarithmic improvement on the $L^p$ norms of eigenfunctions, with the same assumptions on the manifold as in Bérard's theorem. According to their proof, the reason why they need non-positive curvature in dimension $n>2$ is that they use lemma \ref{lemma:old}. Their result can thus be generalized to all closed manifolds without conjugate points.
\end{remark}

To understand the proof, we need to introduce the $\Theta$ function, announced in the title: for $x,x' \in N$
\begin{equation*}
\Theta(x,x') = \det T_{\exp_x^{-1}(x')} \exp_x.
\end{equation*}
As $T_{\exp_x^{-1}(x')} \exp_x$ is a linear application between $T_x N$ and $T_{x'} N$, this determinant is naturally computed taking as reference the volume form $d\vol_g$ at the points $x$ and $x'$. We also define
\begin{equation*}
\vartheta(x,x') = d(x,x')^{n-1} \Theta(x,x').
\end{equation*}
Now, we can give an explicit expression for the coefficients $u_k$ (see \cite{Berard-77}):
\begin{align*}
u_0(x,x') 				&= \frac{1}{\sqrt{\Theta}} \\
u_{k+1}(x,x') 	&= \frac{1}{r^{k+1}\sqrt{\Theta}} \int_0^r s^k \sqrt{\Theta(x,x_s)} (-\Delta_{x'}u_k)(x,x_s) ds.
\end{align*}
where $x' = \exp_x( r u)$ and $x_s = \exp_x(su)$. When investigating the Weyl law for some non-compact manifolds of finite volume with hyperbolic cusps (see \cite{Bonthonneau-4}), it was convenient to introduce a modified version of the Hadamard parametrix. This involved new coefficients $\tilde{u}_k$, which satisfy
\begin{align*}
\tilde{u}_0(x,x') 				&= \sqrt{\frac{\sinh(r)^{n-1}}{\vartheta}} \\
\tilde{u}_{k+1}(x,x') 	&= \frac{1}{\sinh(r)}\int_0^r \left(\frac{\sinh(s)}{\sinh(r)}\right)^{k - \frac{n-1}{2}} \sqrt{\frac{\vartheta(x,x_s)}{\vartheta(x,x')}} \left(-\Delta_{x'} + k^2 - n +1\right)\tilde{u}_k(x,x_s) ds,
\end{align*}
still with $r= d(x,x')$. The proof of lemma \ref{lemma:old} which can be found in the appendix of Bérard's article is sufficiently robust so that we can make two remarks:
\begin{itemize}
	\item The same proof with the same result applies to the coefficients $\tilde{u}_k$.
	\item To complete the proof in the general case, it suffices to prove the most basic estimate, that is, $u_0 = \mathcal{O}(1)e^{\mathcal{O}(d(x,x'))}$. This fact was actually hinted at in the last remark of Bérard's paper.
\end{itemize}

The proof of lemma \ref{lemma:new} will therefore be complete if we can prove this Riemannian geometry lemma:
\begin{lemma}\label{lemma:lower-bound}
Assume that $(N,g)$ is a complete simply-connected manifold without conjugate points, and bounded sectional curvature. For all $\epsilon>0$, there is a constant $C>0$ so that $\vartheta(x,x') > C$ whenever $d(x,x')>\epsilon$.
\end{lemma}

The rest of this note is devoted to the proof of this lemma. For surfaces, it is due to Green (see lemma 2 in \cite{Green-56}). While it may have been known for a while, we did not find any published statement, or proof, for the general case. In Eberlein \cite{Eberlein-73}, one can find a proof that for any $(x,u)$, $\lim_{t\to \infty} \vartheta(x,\exp_x(tu)) = + \infty$, but the convergence is not uniform in $(x,u)$ in higher dimension, so that this is not enough to deduce lemma \ref{lemma:lower-bound}. As a special case, Goto \cite{Goto-78} proved the lemma for manifolds with no \emph{focal points}.

\textbf{Acknowledgement} It was a great pleasure to discuss the matters of this note with Pierre H. Bérard. He communicated a note \cite{Berard-16} with a different take on the proof of lemma \ref{lemma:lower-bound}, via Bochner's formula.

\section{Proof}

The arguments we use are somewhat elementary, and they are inspired by the original proof of Green \cite{Green-58}, and some arguments from Eberlein \cite{Eberlein-73}. However, our proof is (almost) self-contained. 

There is a direct link between the $\Theta$ function and Jacobi fields. Let us fix for the moment a geodesic $\gamma(t)$, starting at $x$, of the form $\exp_x(tu)$. Then we choose a direct orthonormal basis in $T_x N$, whose last vector is $u$. Using parallel transport, this defines a family of parallel direct orthonormal frames along $\gamma$, and we can express Jacobi fields in those frames. They are found to satisfy the usual \emph{matrix} equation:
\begin{equation}\label{eq:Jacobi}\tag{$\ast$}
X''(t) + \mathbb{K}(t) X(t) = 0
\end{equation}
where $\mathbb{K} (t) X(t) = R_{\gamma(t)}(X(t), \dot{\gamma}(t)) \dot{\gamma}(t)$, $R$ being the curvature tensor of $N$. In particular, $\mathbb{K}$ is a symmetric matrix, and $\mathbb{K}(t)\dot{\gamma} = 0$, so $\mathbb{K}(t)$ preserves the orthogonal of $\dot{\gamma}$. Hence we can decompose the Jacobi fields into a parallel part $c(t) \dot{\gamma}(t)$, and an orthogonal part. The parallel coefficient $c(t)$ is of the form $a t + b$.

In these coordinates, the matrix for $T_{t u} \exp_x$ is $t^{-1}\mathbb{A}(t)$ where $\mathbb{A}(t)$ is the Jacobi matrix field such that $\mathbb{A}(0)=0$ and $\mathbb{A}'(0) = \mathbb{1}$. If we decompose this into parallel and orthogonal fields, the parallel part is of course $1$, so we can abuse notations, and still denote by $\mathbb{A}(t)$ the orthogonal field. In what follows, we will only deal with orthogonal fields.

With the notations above, 
\begin{equation*}
\Theta(x, \exp_x(tu)) = \frac{det(\mathbb{A}(t))}{t^{n-1}}.
\end{equation*}
The condition that there are no conjugate points is equivalent to assuming that the field $\mathbb{A}(t)$ is invertible for $t\neq 0$, independently of the vector $(x,u)$ (or equivalently, that for all $x$, $\exp_x$ is a global diffeomorphism). We also have
\begin{equation*}
\vartheta(x, \exp_x(tu)) = \det \mathbb{A}(t).
\end{equation*}
We will use the Ricatti equation associated to \eqref{eq:Jacobi}, that is
\begin{equation}\label{eq:Ricatti}\tag{$\ast \ast$}
\mathbb{V}' + \mathbb{V}^2 + \mathbb{K} = 0.
\end{equation}
This is also a matrix-valued equation along $\gamma(t)$, and it is satisfied for $\mathbb{V}$'s of the form $\mathbb{B}' \mathbb{B}^{-1}$, where $\mathbb{B}$ is an invertible solution of \eqref{eq:Jacobi}. The non-conjugacy assumption will imply the existence of solutions to \eqref{eq:Ricatti} on the interval $[0,+\infty)$, and this can be seen as the conceptual argument behind the proof.

The following lemma is fundamental for our argument. It can be found in Green for surfaces (see lemma 3 in \cite{Green-58}), or in Eberlein in this level of generality (lemma 2.8 in \cite{Eberlein-73}). It is actually a generalization of a result on Sturm-Liouville equations, known at least since E. Hopf. When we write an inequality between two matrices, they are assumed to be symmetric, and it means that the corresponding inequality holds between the associated quadratic forms.
\begin{lemma}\label{lemma:Ricatti}
Let $\mathbb{V}$ be a symmetric solution to the Ricatti equation \eqref{eq:Ricatti}, defined for all $t>0$, then $|\mathbb{V}| \leq k \coth k t$ as soon as $\mathbb{K} \geq - k^2 \mathbb{1}$ for all $t>0$.
\end{lemma}

Before going any further, let us make two remarks

\begin{enumerate}
	\item It is useful to recall that if $\mathbb{B}$ and $\mathbb{C}$ are two solutions of \eqref{eq:Jacobi}, their Wronskian is $\mathbb{W}(\mathbb{B},\mathbb{C}) = \mathbb{B}^\ast \mathbb{C}' - \mathbb{B}'^\ast \mathbb{C}$ --- here $L^\ast$ is the transpose of $L$. It is constant. In particular, if $\mathbb{W}(\mathbb{B}, \mathbb{B}) = 0$ and $\mathbb{B}$ is invertible, the associated solution  $\mathbb{B}'\mathbb{B}^{-1}$ of \eqref{eq:Ricatti} is symmetric.
	\item Let $\mathbb{U} = \mathbb{A}' \mathbb{A}^{-1}$. One can check that it \emph{is} a symmetric solution to \eqref{eq:Ricatti}. Additionally, we find
\begin{equation*}
\frac{d}{dt} \vartheta(x,\exp_x(tu)) = \vartheta(x,\exp_x(tu)) \Tr \mathbb{U}(t).
\end{equation*}
In particular, this implies the existence of a bound of the form $| \vartheta(t)|^{-1} = \mathcal{O}(t^{1-n})e^{\mathcal{O}(t)}$, where the constants are independent of $x$ and $u$. This would probably be sufficient to obtain lemma \ref{lemma:new}, but we will nonetheless go on with the proof of lemma \ref{lemma:lower-bound}.
\end{enumerate}

\subsection{Green's method}

In this section, we recall the proof of existence of Green’s bundles. As a corollary of the proof, we find a bound that is exactly what we need. Let us introduce further notations. The vector $(x,u)$ is still fixed.

For $t\neq 0$, the field $\mathbb{D}_t$ is the unique solution to \eqref{eq:Jacobi} such that $\mathbb{D}_t(0) = \mathbb{1}$, and $\mathbb{D}_t(t) = 0$. The existence and unicity of such a field is assured by the non-conjugacy assumption. Green’s bundles are the subbundles $X_{\pm}$ in $TSN$ such that ${X_{\pm}}_{(x,u)} = \{ (\xi, \mathbb{D}_{\pm \infty}'(0)\cdot \xi), \xi \in T_x N \}$. Here $\mathbb{D}_{\pm \infty}$ is the limit of the fields $\mathbb{D}_t$ as $t\to \pm \infty$. Let us recall the proof of existence of such limits. 

First, there is an explicit expression for $\mathbb{D}_t$:
\begin{equation*}
\mathbb{D}_t(s) = \mathbb{A}(s) \int_s^t \mathbb{A}(\ell)^{-1}\mathbb{A}(\ell)^{-1 \ast}d\ell.
\end{equation*}
This is valid when $s$ and $t$ have the same sign. Observe that, still if $s$ and $t$ have the same sign,
\begin{equation*}
\mathbb{D}_t'(0) - \mathbb{D}_s'(0) = \int_s^t \mathbb{A}(\ell)^{-1} \mathbb{A}(\ell)^{-1 \ast} d\ell.
\end{equation*}
In particular, the matrices $\mathbb{D}_t'(0) - \mathbb{D}_s'(0)$ are symmetric, and as symmetric matrices, with $s>0$ fixed, the family $t\mapsto \mathbb{D}_t'(0) - \mathbb{D}_s'(0)$ is increasing. We want to show that it has a limit. It suffices to prove that it is bounded from above.

Since for any $t>0$, $\mathbb{A}$ et $\mathbb{D}_t$ are two independent solutions of \eqref{eq:Jacobi}, we can write
\begin{equation*}
\mathbb{D}_{-s} = \mathbb{A} \mathbb{N}_{s,t} + \mathbb{D}_t
\end{equation*}
where $\mathbb{N}_{s,t}$ is a constant matrix. We find
\begin{equation}\label{eq:pre-borne}
\mathbb{D}_{-s}'(0) - \mathbb{D}_t'(0) = \mathbb{N}_{s,t}.
\end{equation}
But we also have
\begin{equation*}
\mathbb{N}_{s,t} = - \mathbb{A}(-s)^{-1} \mathbb{D}_t(-s)
\end{equation*}
This implies that $N_{s,t}$ is symmetric positive definite for all $s,t>0$. Indeed, $\mathbb{N}_{-t,t} = 0$, and for $s\neq 0$,
\begin{equation*}
\frac{d}{ds}(\mathbb{N}_{s,t} - \mathbb{N}_{s,t}^\ast) = \mathbb{A}^{-1}\mathbb{A}^{-1 \ast}(s) \mathbb{W}(\mathbb{A}, \mathbb{D}_t) + \mathbb{W}(\mathbb{D}_t, \mathbb{A})\mathbb{A}^{-1}\mathbb{A}^{-1 \ast}(s) = 0.
\end{equation*}
This proves that $N_{s,t}$ is symmetric for all $s<0$. Now, examining $s\to 0$, we find that $\mathbb{N} - \mathbb{N}^\ast$ is continuous at $s=0$, so that $N_{s,t}^\ast = N_{s,t}$ for all $s$. To check that $\mathbb{N}$ is positive definite, it suffices to remark, as Green did, that for $s>0$ small enough, it is true because $\mathbb{N}_{s,t} \sim (1/s) \mathbb{1}$. Next, for it to stop being true, the determinant would need to vanish at some $s>0$. But $\mathbb{A}(-s)$ and $\mathbb{D}_t(-s)$ are invertible for $s>0$, so this is not possible.

As a consequence, for $t>0$ and $s>0$, according to \eqref{eq:pre-borne},
\begin{equation*}
\mathbb{D}_t'(0) - \mathbb{D}_{-s}'(0)  \leq 0
\end{equation*}
and so for $s>0$, and $t>s$,
\begin{equation*}
\mathbb{D}_t'(0) - \mathbb{D}_s'(0) \leq \mathbb{D}_{-s}'(0) - \mathbb{D}_s'(0).
\end{equation*}
Letting $t$ go to $+\infty$, we find:
\begin{equation*}
M(s) : =\lim_{t\to +\infty} \mathbb{D}_t'(0) - \mathbb{D}_s'(0) = \int_s^{+\infty} \mathbb{A}(\ell)^{-1}\mathbb{A}(\ell)^{-1 \ast} d\ell < \infty.
\end{equation*}
As $M(s)$ is decreasing, for $t>s$,
\begin{equation*}
M(t) \leq \mathbb{D}_{-s}'(0) - \mathbb{D}_s'(0).
\end{equation*}
We also find that for $s>0$,
\begin{equation}\label{eq:D-infty}
\mathbb{D}_{+\infty}(s) = \mathbb{A}(s) \int_s^{+\infty} \mathbb{A}(\ell)^{-1}\mathbb{A}(\ell)^{-1 \ast} d\ell.
\end{equation}

\subsection{End of the proof}

From formula \eqref{eq:D-infty}, we deduce that $\mathbb{D}_{+\infty}(s)$ is invertible for every $s\geq 0$. Hence the associated solution of \eqref{eq:Ricatti} $\mathbb{V} = \mathbb{D}_{+\infty}'\mathbb{D}_{+\infty}^{-1}$ is defined at least for all $s\geq 0$. It is also symmetric. Indeed, for any $t$, we have $\mathbb{W}( \mathbb{D}_t, \mathbb{D}_t )=0$ by evaluation at $t$. Letting $t\to + \infty$, we get $\mathbb{W}( \mathbb{D}_{+\infty}, \mathbb{D}_{+\infty} )=0$.

Now, we want to apply lemma \ref{lemma:Ricatti} on solutions of \eqref{eq:Ricatti}. We consider
\begin{equation*}
\mathbb{U}(t) - \mathbb{V}(t) = \mathbb{A}(t)^{-1 \ast} M(t)^{-1} \mathbb{A}(t)^{-1}.
\end{equation*}
In particular,
\begin{equation*}
\left|\mathbb{A}(t)^{-1 \ast} M(t)^{-1} \mathbb{A}(t)^{-1} \right| \leq 2 k \coth kt.
\end{equation*}
As $M(s)$ is symmetric, positive definite, $M^{-1}(s) \geq 1/\| M(s)\|_2$, and
\begin{equation*}
2\| M(t)\|_2 k \coth kt \geq \mathbb{A}(t)^{-1 \ast}\mathbb{A}(t)^{-1}
\end{equation*}
As a consequence,
\begin{equation*}
\| \mathbb{A}^{-1} (t)\|_2^2  \leq 2\| M(t)\|_2 k \coth kt
\end{equation*}
We finally get the result of this computation: for $t>s>0$,
\begin{equation*}
\| \mathbb{A}(t)^{-1} \|_2^2  \leq 2 k \coth kt \| \mathbb{D}_{-s}'(0) - \mathbb{D}_s'(0)\|_2
\end{equation*}
and according to Hadamard's inequality,
\begin{equation*}
\vartheta^{-1}(t) \leq \left\{2 k \coth kt \| \mathbb{D}_{-s}'(0) - \mathbb{D}_s'(0)\|_2\right\}^{\frac{n-1}{2}}
\end{equation*}

Now, the last step of the proof is 
\begin{lemma}
The quantity $\| \mathbb{D}_{-s}'(0) - \mathbb{D}_s'(0)\|_2$ is bounded independently of $x$ and $u$, provided that $s>0$ is small enough. For this, it suffices to assume that the sectional curvature of $(N,g)$ is bounded. 
\end{lemma}

\begin{proof}
Let us fix a vector $(x,u)\in SN$. Let $\mathbb{J}_1$ and $\mathbb{J}_2$ be the two Jacobi matrix fields along $\{\exp_x (tu)\}$ such that
\begin{equation*}
\mathbb{J}_1(0)= \mathbb{1},\ \mathbb{J}_1'(0)= 0,\ \mathbb{J}_2(0)= 0,\ \mathbb{J}_2'(0)= \mathbb{1}.
\end{equation*}
One can check that for $s\neq 0$, $\mathbb{D}'_s(0) = - \mathbb{J}_2(s)^{-1}\mathbb{J}_1(s)$. Whence we deduce that as $s\to 0$,
\begin{equation*}
\mathbb{D}'_s(0) = - \frac{1}{s} \mathbb{1} + \mathcal{O}(s).
\end{equation*}
To prove the lemma, it suffices to prove that the constant in $\mathcal{O}(s)$ can be controlled by the sectional curvature. Let $k_{max}$ be the supremum of $\| \mathbb{K}_{x,u}(t) \|$ for all $x,u,t$. This is actually the supremum of the absolute value of the sectional curvature of $N$. 

Recall from the proof of the Cauchy-Lipschitz theorem that the Jacobi fields $\mathbb{J}_1$ and $\mathbb{J}_2$ are obtained, at least for small times, as fixed point of contraction mappings, respectively
\begin{equation*}
T_1^\tau J(t) = \mathbb{1} + \int_{0}^t (s-t) \mathbb{K}(s)J(s) ds,\ \text{and}\quad T_2^\tau J(t) = t\cdot \mathbb{1} + \int_{0}^t (s-t) \mathbb{K}(s)J(s) ds,
\end{equation*}
which are defined on $C^0$ matrix-valued functions on $t \in [-\tau, \tau]$, equipped with the norm $\| J \|_\tau = \sup_{|t|\leq \tau} \|J(t)\|_2$. These mappings have Lipschitz constant $\eta :=\tau^2 k_{max}/2$, so we take $0< \tau < \sqrt{2/k_{max}}$. From the Banach fixed point theorem, we know that $\mathbb{J}_1 = \lim_n (T^\tau_1)^n(\mathbb{1})$ and $\mathbb{J}_2 = \lim_n (T^\tau_2)^n(t\cdot\mathbb{1})$, and using usual tricks,
\begin{equation*}
\| \mathbb{J}_1 - \mathbb{1} \|_\tau \leq \frac{1}{1-\eta} \| T^\tau_1(\mathbb{1}) - \mathbb{1} \|_\tau,\ \text{and} \quad \| \mathbb{J}_2 - t\cdot\mathbb{1} \|_\tau \leq \frac{1}{1-\eta} \| T^\tau_2(t\cdot \mathbb{1}) - t\cdot\mathbb{1} \|_\tau
\end{equation*}
For $\tau$ small enough, we find that with a constant $C_{max} >0$ only depending on $k_{max}$,
\begin{equation*}
\| \mathbb{J}_1(\tau) - \mathbb{1} \|_2 \leq C_{max} \tau^2,\ \text{and}\quad \| \mathbb{J}_2(\tau) - \tau\cdot\mathbb{1} \|_2 \leq C_{max} \tau^3.
\end{equation*}
\end{proof}

\bibliographystyle{alpha}
\bibliography{../../bibliographies/biblio}

\def\dbar{\leavevmode\hbox to 0pt{\hskip.2ex \accent"16\hss}d} \def\cprime{$'$}
  \def\cprime{$'$}
\begin{thebibliography}{{Bon}15}

\bibitem[BBB87]{Ballmann-Brin-Burns-87}
W.~Ballmann, M.~Brin, and K.~Burns.
\newblock On surfaces with no conjugate points.
\newblock {\em J. Differential Geom.}, 25(2):249--273, 1987.

\bibitem[B{\'e}r77]{Berard-77}
Pierre~H. B{\'e}rard.
\newblock On the wave equation on a compact {R}iemannian manifold without
  conjugate points.
\newblock {\em Math. Z.}, 155(3):249--276, 1977.

\bibitem[B{\'e}r16]{Berard-16}
Pierre B{\'e}rard.
\newblock Private communication.
\newblock March 2016.

\bibitem[{Bon}15]{Bonthonneau-4}
Y.~{Bonthonneau}.
\newblock {Weyl laws for manifolds with hyperbolic cusps}.
\newblock {\em \href{http://arxiv.org/abs/1512.05794}{ArXiv 1512.05794}},
  December 2015.

\bibitem[Ebe73]{Eberlein-73}
Patrick Eberlein.
\newblock When is a geodesic flow of {A}nosov type? {I},{II}.
\newblock {\em J. Differential Geometry}, 8:437--463; ibid. 8 (1973), 565--577,
  1973.

\bibitem[Got78]{Goto-78}
Midori~S. Goto.
\newblock Manifolds without focal points.
\newblock {\em J. Differential Geom.}, 13(3):341--359, 1978.

\bibitem[Gre56]{Green-56}
L.~W. Green.
\newblock Geodesic instability.
\newblock {\em Proceedings of the American Mathematical Society},
  7(3):438--448, 1956.

\bibitem[Gre58]{Green-58}
L.~W. Green.
\newblock A theorem of {E}. {H}opf.
\newblock {\em Michigan Math. J.}, 5:31--34, 1958.

\bibitem[Gul75]{Gulliver-75}
Robert Gulliver.
\newblock On the variety of manifolds without conjugate points.
\newblock {\em Trans. Amer. Math. Soc.}, 210:185--201, 1975.

\bibitem[HT15]{Hassell-Tacy-15}
Andrew Hassell and Melissa Tacy.
\newblock Improvement of eigenfunction estimates on manifolds of nonpositive
  curvature.
\newblock {\em Forum Mathematicum}, 27(3):1435--1451, 2015.

\end{thebibliography}
\end{document}